\documentclass[12pt,a4paper]{amsart}
\usepackage[all,cmtip]{xy}
\usepackage{enumerate, comment}
\usepackage{ amssymb, latexsym, amsmath}
\usepackage{fullpage}
\usepackage{url}
\usepackage{hyperref}
\usepackage{helvet}
\usepackage{amsfonts}
\usepackage{tikz-cd}
\usepackage{amssymb}
\usepackage{amsthm}
\usepackage{amsmath}
\usepackage{amscd}
\usepackage[mathscr]{eucal}
\usepackage{indentfirst}
\usepackage{graphicx}
\usepackage{graphics}
\usepackage{pict2e}
\usepackage{epic}
\usepackage[OT2,T1]{fontenc}
\numberwithin{equation}{section}
\usepackage[margin=2.9cm]{geometry}

\theoremstyle{plain}

\newtheorem{Th}{Theorem}[section]
\newtheorem{Lemma}[Th]{Lemma}

\newtheorem{question}[Th]{Question}
 \theoremstyle{definition}

\newtheorem{Def}[Th]{Definition}
\newtheorem{Fact}[Th]{Fact}
\newtheorem{Convention}[Th]{Convention}

\makeatletter
\newcommand*{\rom}[1]{\expandafter\@slowromancap\romannumeral #1@}
\makeatother
\newcommand{\Q}{\mathbb{Q}}

\newcommand{\Z}{\mathbb{Z}}

\newcommand{\g}{\operatorname{Ad}^0\bar{\rho}}
\newcommand{\ad}{\operatorname{Ad}\bar{\rho}}

\newcommand{\GL}{\text{GL}}

\newcommand\op[1]{\operatorname{#1}}
\newcommand{\F}{\mathbb{F}}

\DeclareSymbolFont{cyrletters}{OT2}{wncyr}{m}{n}

\DeclareFontFamily{U}{wncy}{}
    \DeclareFontShape{U}{wncy}{m}{n}{<->wncyr10}{}
    \DeclareSymbolFont{mcy}{U}{wncy}{m}{n}
    \DeclareMathSymbol{\Sh}{\mathord}{mcy}{"58}

\title{Constructing Galois Representations Ramified at one prime}
\author{Anwesh Ray}

\begin{document}

\maketitle
\begin{abstract} Let $n>1$, $e\geq 0$ and a prime number $p\geq 2^{n+2+2e}+3$, such that the index of regularity of $p$ is $\leq e$. We show that there are infinitely many irreducible Galois representations $\rho: \op{Gal}(\bar{\Q}/\Q)\rightarrow \op{GL}_n(\Q_p)$ unramified at all primes $l\neq p$. Furthermore, these representations are shown to have image containing a fixed finite index subgroup of $\op{SL}_n(\Z_p)$. Such representations are constructed by lifting suitable residual representations $\bar{\rho}$ with image in the diagonal torus in $\op{GL}_n(\F_p)$, for which the global deformation problem is unobstructed.
\end{abstract}
\section{Introduction}
\par Let $p$ be an odd prime and $n$ an integer. There is much interest in the study of continuous Galois representations 
\[\rho:\op{G}_{\Q}\rightarrow \op{GL}_n(\bar{\Q}_p)\]which are \textit{geometric}, in the sense of Fontaine-Mazur (cf. \cite{FonaineMazur}). Prototypical examples include Galois representations associated to abelian varieties over $\Q$ and Siegel modular forms. The Galois representations associated to abelian varieties have $p$-adic monodromy $\op{GSp}_{2g}(\Q_p)$. On the other hand, Galois representations with big image in $\op{GL}_n(\Z_p)$ for $n>2$ are not expected to arise from automorphic forms. 
\par In this article we study the following question:
\begin{question}\label{question}
Let $p$ be a prime and $n>1$. Does there exist a continuous Galois representation $\rho:\op{G}_{\Q}\rightarrow \op{GL}_n(\Z_p)$ with suitably large image? If so, can one control the set of primes at which it may ramify?
\end{question}A parallel may be seen in the study of geometric Galois representations, where in the $\op{GL}_2$ case, the ramification may be controlled via Ribet's level lowering argument. For certain $(p,n)$, Greenberg systematically constructed Galois representations $\op{G}_{\Q,\{p\}}\rightarrow \GL_n(\Z_p)$ with image containing a finite index subgroup of $\op{SL}_n(\Z_p)$. Let $M$ be the maximal pro-$p$ extension of $\Q(\mu_p)$ which is unramified outside $p$. A theorem of Shafarevich shows that is $p$ is a regular prime, then $\op{Gal}(M/\Q)$ is a free pro-$p$ group with $\frac{p+1}{2}$ generators. Greenberg makes use of this to construct such Galois representations $\op{G}_{\Q,\{p\}}\rightarrow \GL_n(\Z_p)$ when $p$ is a regular prime greater than or equal to $4\lfloor n/2\rfloor +1$ (see \cite[Proposition 1.1]{greenberg}). Cornut-Ray \cite{raycornut} further generalized Greenberg's results to more general algebraic groups, without relaxing the regularity assumption on $p$. In \cite{tang}, Tang relaxed the regularity assumption of Greenberg, by constructing certain mod-$p$ representations which lift to characteristic zero when they are allowed to ramify at an auxiliary set of primes. This relies on deformation techniques pioneered by Ramakrishna \cite{Ram2,RaviFM}. Thus, Tang provides an affirmative answer to the first part of Question $\ref{question}$.
Our goal in this article is to control ramification. The residual representation is chosen to be unramified away from $p$ so that the associated global deformation problem is \textit{unobstructed}. This allows us to produce characteristic zero lifts without adding further ramification away from $p$. The construction in this paper is brief and self contained, relying only on well-known results in Galois cohomology. Let $e_p$ denote the index of regularity, see Definition $\ref{index}$.
\begin{Th}\label{main} Let $n>1$, $e\geq 0$ and $p$ be a prime number such that 
  \begin{enumerate}
      \item $p\geq  2^{n+2+2e}+3$,
      \item the index of regularity $e_p\leq e$.
  \end{enumerate}
      There are infinitely many continuous representations $\rho:\op{G}_{\Q,\{p\}}\rightarrow \op{GL}_n(\Z_p)$ such that the image of $\rho$ contains $\op{ker}\left(\op{SL}_n(\Z_p)\rightarrow \op{SL}_n(\Z/p^4)\right)$.
  \end{Th}
  It is noted, for instance in \cite{buhler}, that if the numerators of Bernoulli numbers are uniformly random modulo odd primes, then the density of irregular primes with index of regularity equal to $ r$ should equal $e^{-1/2} /(2^r r!)$. This heuristic is supported by evidence, indeed, it is shown in \textit{loc. cit.} that among the first million primes, the highest index of regularity observed is $6$, and the only prime less than a million with index of regularity equal to $6$ is $527377$. The density $e^{-1/2} /(2^r r!)$ drops rather fast. The density of regular primes is expected to be $e^{-1/2}=0.60653065...$, whilst the density of irregular primes with index of regularity $6$ is expected to be $0.00001316...$. In \cite{hart}, it is shown that the maximum regularity index for primes $p<2^{31}$ is $9$. Also note that it is known that Vandiver's conjecture is satisfied for all primes less than $2^{19}$. Specializing the above to primes less than $2^{31}$, we have the following:
  \begin{Th}
    Let $n$ be such that $2\leq n\leq 10$ and $p$ a prime such that $2^{n+20}<p< 2^{31}$. 
    There are infinitely many continuous representations $\rho:\op{G}_{\Q,\{p\}}\rightarrow \op{GL}_n(\Z_p)$ such that the image of $\rho$ contains $\op{ker}\left(\op{SL}_n(\Z_p)\rightarrow \op{SL}_n(\Z/p^4)\right)$.
  \end{Th}
  When $n$ and $p$ are specified, it is indeed possible to check if the method in this article can be used to construct a Galois representation. One may try and construct the sequence $k_1,\dots, k_n$ satisfying the properties of Theorem $\ref{main2}$. However, the author was not able to realize a more refined statement which applies in suitable generality, than Theorem $\ref{main}$. It is natural to ask if the results in this manuscript can be generalized to split reductive algebraic groups over $\Z_p$, as is done in \cite{raycornut,tang}. The author intentionally chooses a less general framework in which the inherent simplicity of the underlying ideas come across easily.
  \subsection*{Acknowledgements}
The author would like to thank Ravi Ramakrishna for helpful conversations. The author is very grateful to the anonymous referee for her/his attentive and timely reading of the original submission.
\section{The global deformation problem}
\par In this section, we introduce some preliminary notions. Fix an odd prime $p$ and a number $n>1$. For each prime number $l$, denote by $\op{G}_{\Q_l}$ the absolute Galois group of $\Q_l$. Choosing an embedding $\bar{\Q}\rightarrow \bar{\Q}_l$, we have an inclusion $\op{G}_{\Q_l}\hookrightarrow \op{G}_{\Q}$ of Galois groups. Denote by $\chi$ the $p$-adic cyclotomic character and $\bar{\chi}$ its mod-$p$ reduction. For $m\geq 1$, denote by $\mathcal{U}_m\subseteq \op{SL}_n(\Z_p)$ the kernel of the mod-$p^m$ reduction map.
\par Fix a sequence of integers $k_1,k_2,\dots, k_n$ and set $\bar{\rho}$ to denote the mod-$p$ Galois representation which is a direct sum of characters $\bar{\chi}^{k_1}\oplus\dots\oplus\bar{\chi}^{k_n}$. In other words, we have the residual representation
\[\bar{\rho}=\left( {\begin{array}{cccc}
   \bar{\chi}^{k_1} & & & \\
    & \bar{\chi}^{k_2}& & \\
    & & \ddots & \\
    & & & \bar{\chi}^{k_{n}}
  \end{array} } \right):\op{G}_{\Q,\{p\}}\rightarrow \op{GL}_n(\F_p).\]
  In order to produce characteristic zero lifts of $\bar{\rho}$ with big image, we study the deformations of $\bar{\rho}$. For an introduction to the deformation theory of Galois representations, the reader may consult \cite{Mazur}.
  \par For a local ring $R$ with maximal ideal $\mathfrak{m}_R$, let $\widehat{\op{GL}_n}(R)$ be the group
  \[\widehat{\op{GL}_n}(R):=\op{ker}\left\{\op{GL}_n(R)\xrightarrow{\op{mod}\mathfrak{m}_R} \op{GL}_n(R/\mathfrak{m}_R)\right\}.\]
\begin{Def}Let $m$ be an integer such that $m\geq 1$. A mod-$p^m$ lift of $\bar{\rho}$ is a continuous homomorphism $\rho_m: \op{G}_{\Q}\rightarrow \op{GL}_n(\Z/p^m)$ such that $\bar{\rho}=\rho_m\mod{p}$. Two lifts $\rho_m, \rho_m': \op{G}_{\Q}\rightarrow \op{GL}_n(\Z/p^m)$ of $\bar{\rho}$ are \textit{strictly equivalent} if $\rho_m'=A\rho_m A^{-1}$ for some matrix $A\in \widehat{\op{GL}_n}(\Z/p^m)$.
A \textit{deformation} is a strict equivalence class of lifts.
\end{Def}
It was shown by Mazur that the global deformation functors associated to absolutely irreducible mod-$p$ Galois representations are indeed representable by \textit{universal deformations}. In this article, $\bar{\rho}$ is far from irreducible. We adopt a step by step lifting approach which does not rely on the existence of a universal deformation. 
\par Let $\tau$ denote the determinant of $\bar{\rho}$ and $\tilde{\tau}:\op{G}_{\Q,\{p\}}\rightarrow \op{GL}_1(\Z_p)$ denote the Teichm\" uller lift of $\tau$. For any character $\phi:\op{G}_{\Q}\rightarrow \op{GL}_1(\Z_p)$, let $\phi_m$ denote the mod-$p^m$ reduction of $\phi$. When there is no cause for confusion, we shall simply use $\phi$ in place of $\phi_m$. Fix a character $\psi$ which is unramified outside $\{p\}$ such that $\psi_2=(\tilde{\tau})_2$. For instance, $\psi$ can be taken to be $\tilde{\tau}\chi^{p^2-p}$.
\begin{Convention}Let us note once and for all that all deformations of $\bar{\rho}$ are stipulated to have determinant equal to $\psi$.
\end{Convention}
\par In order to prove Theorem $\ref{main}$, it is shown that for a suitable choice of $k_1,\dots, k_{n}$ it is shown that $\bar{\rho}$ lifts to a characteristic zero irreducible representation which is unramified at all primes $l\neq p$. 
 \begin{enumerate}
     \item First, it is shown that there is a mod-$p^5$ lift $\rho_5: \op{G}_{\Q,\{p\}}\rightarrow \op{GL}_n(\Z/p^5)$ such that the image of $\rho_5$ contains \[\op{ker}\left\{\op{SL}_n(\Z/p^5)\rightarrow \op{SL}_n(\Z/p^4)\right\}. \]
     \item Next, it is shown that the (unramified outside $\{p\}$) infinitesimal lifting problem is unobstructed. This implies that any mod-$p^m$ deformation
     \[\rho_{m}: \op{G}_{\Q, \{p\}}\rightarrow \op{GL}_n(\Z/p^m)\] of $\bar{\rho}$ lifts one more step as depicted:
     \[\begin{tikzpicture}[node distance = 2.0 cm, auto]
      \node (GSX) at (0,0){$\operatorname{G}_{\Q,\{p\}}$};
      \node (GL2) at (5,0){$\op{GL}_n(\F_p).$};
      \node (GL2Wn) at (3,2)[above of= GL2]{$\op{GL}_n(\Z/p^{m})$};
      \node (GL2Wnplus1) at (5,4){$\op{GL}_{n}(\Z/p^{m+1})$};
      \draw[->] (GSX) to node [swap]{$\bar{\rho}$} (GL2);
      \draw[->] (GL2Wn) to node {} (GL2);
      \draw[->] (GSX) to node [swap]{$\rho_m$} (GL2Wn);
      \draw[->] (GL2Wnplus1) to node {} (GL2Wn);
      \draw[dashed,->] (GSX) to node {$\rho_{m+1}$} (GL2Wnplus1);
      \end{tikzpicture}\] It follows that $\rho_5$ lifts to a characteristic zero continuous representation $\rho:\op{G}_{\Q,\{p\}}\rightarrow \op{GL}_n(\Z_p)$.
\item It is shown that the image of ${\rho}$ contains $\mathcal{U}_4$.
 \end{enumerate}
\par Let us describe the infinitesimal deformation problem.
\begin{Def}\label{adgaloisaction}Set $\ad$ to denote the Galois module whose underlying vector space consists of $n\times n$ matrices with entries in $\F_p$. Let $\g$ be the Galois stable submodule of trace zero matrices. The Galois action is as follows: for $g\in \op{G}_{\Q,\{p\}}$ and $v\in \ad$, 
     set $g\cdot v:=\bar{\rho}(g) v \bar{\rho}(g)^{-1}$.
     \end{Def}The module $\ad$ is equipped with a Lie bracket $[X,Y]:=XY-YX$. The underlying vector space of $\ad$ (resp. $\g$) is the Lie algebra of $\op{GL}_{n/\F_p}$ (resp. $\op{SL}_{n/\F_p}$). Let $e_{i,j}\in \g$ denote the matrix with $1$ at the $(i,j)$ entry and $0$ at all other entries. The adjoint Galois-module $\g$ is a direct sum 
     \[\g=\mathfrak{t}\oplus\left(\bigoplus_{(i,j),
     i\neq j}\F_p(\bar{\chi}^{k_i-k_j})\right), \]where $\mathfrak{t}$ is the submodule of diagonal matrices and the sum runs over $(i,j)$ with $i\neq j$. The Galois action on $\mathfrak{t}$ is trivial.
     \par Suppose that $m\geq 1$ and $\rho_{m}:\op{G}_{\Q,\{p\}}\rightarrow \op{GL}_n(\Z/p^{m})$ is a deformation of $\bar{\rho}$. A continuous lift (not necessarily a homomorphism) $\varrho: \op{G}_{\Q,\{p\}}\rightarrow \op{GL}_n(\Z/p^{m+1})$ of $\rho_{m}$ with determinant $\psi_{m+1}$ does always exist. Identify $\g$ with the kernel of the mod-$p^m$ map $\op{SL}_n(\Z/p^{m+1})\rightarrow \op{SL}(\Z/p^m)$ by associating a vector $X\in \g$ with $\op{Id}+p^m X$. Let $\mathcal{O}(\rho_{m})$ be the cohomology class in $H^2(\op{G}_{\Q,\{p\}}, \g)$ defined by the $2$-cocycle 
     \[(g,h)\mapsto\varrho(gh)\varrho(h)^{-1} \varrho(g)^{-1}.\]The associated cohomology class $\mathcal{O}(\rho_{m})$ so defined is independent of the lift $\varrho$. The following is easy to check.
     \begin{Fact} A mod-$p^{m}$ deformation $\rho_{m}$ does lift one more step to a Galois representations which is unramified outside $\{p\}$
     \[\rho_{m+1}:\op{G}_{\Q,\{p\}}\rightarrow \op{GL}_n(\Z/p^{m+1}) \]if and only if $\mathcal{O}(\rho_m)=0$.
     \end{Fact}
     The next fact states that the set of deformations $\rho_{m+1}$ of $\rho_{m}$ have the structure of an $H^1(\op{G}_{\Q,\{p\}}, \g)$-pseudotorsor.
     \begin{Fact}Suppose that there exist deformations $\rho_{m+1},\rho_{m+1}':\op{G}_{\Q,\{p\}}\rightarrow \op{GL}_n(\Z/p^{m+1})$ of $\rho_{m}$. Then there is a unique class $h\in H^1(\op{G}_{\Q,\{p\}}, \g)$ such that \[\rho_{m+1}'=(\op{Id}+p^{m} h)\rho_{m+1}.\]
    
     \end{Fact}
We say that the "unramified outside $\{p\}$"  deformation problem for $\bar{\rho}$ is \textit{unobstructed} if $H^2(\op{G}_{\Q,\{p\}}, \g)$ is equal to zero. For future reference, we take note of the following, which follows from the previous discussion.
\begin{Lemma}\label{lemma25}
Suppose that $H^2(\op{G}_{\Q,\{p\}}, \g)=0$ and suppose that we are given a deformation $\rho_m:\op{G}_{\Q,\{p\}}\rightarrow \op{GL}_n(\Z/p^m)$ of $\bar{\rho}$. There exists a deformation \[\rho: \op{G}_{\Q,\{p\}}\rightarrow \op{GL}_n(\Z_p)\] such that $\rho_m=\rho\mod{p^m}$.
\end{Lemma}
In the next section, appropriate choices of $\bar{\rho}$ are shown to be unobstructed outside $\{p\}$ in the sense described. The results proven in the remainder of this section are used in showing that the characteristic zero lifts thus constructed do indeed have big image in $\op{SL}_n(\Z_p)$. Suppose that $\rho:\op{G}_{\Q,\{p\}}\rightarrow \op{GL}_n(\Z_p)$ is a lift of $\bar{\rho}$. For $m\geq 1$, set $\rho_m$ to be the mod-$p^m$ reduction $\rho\mod{p^m}$. 
\begin{Def}For $m\geq 1$, set $\Phi_m(\rho):=\rho_{m+1} (\op{ker}\rho_{m})$. Note that $\Phi_m(\rho)$ is isomorphic to $\op{ker} \rho_m /\op{ker} \rho_{m+1}$ and shall be viewed as a submodule of $\ad$. Here, $\op{Id}+p^m v=\rho_{m+1}(g)$ for $g\in \op{ker}\rho_m$, is identified with $v\in \ad$.
\end{Def}
 Recall from Definition $\ref{adgaloisaction}$ that the Galois action on $\op{Ad}\bar{\rho}$ is from composing $\bar{\rho}$ with the adjoint action. Note that if $\sigma\in \op{G}_{\Q,\{p\}}$ and $v\in \op{Ad}\bar{\rho}$, then 
 \[\rho_{m+1}(\sigma) (\op{Id}+p^m v)\rho_{m+1}(\sigma)^{-1}=(\op{Id}+p^m \bar{\rho}(\sigma)v\bar{\rho}(\sigma)^{-1})=(\op{Id}+p^m (\sigma\cdot v)).\]It is easy to check that since $\op{ker} \rho_{m}$ is a normal subgroup of $\GL_n(\Z_p)$, it follows that $\Phi_m(\rho)\subseteq \ad$ is a Galois-stable submodule. Recall that the determinant of $\rho$ is stipulated to be equal to the character $\psi$, which is chosen to be congruent to $\tilde{\tau}$ modulo-$p^2$. As a result, for $g\in \op{ker}\bar{\rho}$, it follows that $\op{det}\rho_2(g)=1$. Writing $\rho_2(g)=\op{Id}+pX$, we have that 
\[\op{det} \rho_2(g)=1+p\op{tr} X,\]and hence $\op{tr}X=0$ in $\Z/p\Z$. It thus follows that $\Phi_1(\rho)\subseteq \g$.
 \begin{Lemma} \label{lemma26}
     Let $\rho$ be as above. For $l,m\geq 1$, we have that $[\Phi_l(\rho), \Phi_m(\rho)]\subseteq \Phi_{l+m}(\rho)$.
     \end{Lemma}
     \begin{proof}
     Set $\mathcal{G}_k$ denote the kernel of the mod-$p^k$ map 
     \[\mathcal{G}_k:=\op{ker}\left\{ \op{GL}_n(\Z/p^{k+1}))\rightarrow\op{GL}_n(\Z/p^{k}) \right\}.\]Let $c\in \Phi_l(\rho)$ and $d\in \Phi_m(\rho)$, consider the elements $\op{Id}+ p^lc\in \mathcal{G}_l$ and $\op{Id}+ p^{m}d\in \mathcal{G}_m$. Let $\tilde{c}, \tilde{d}$ be such that $A=\op{Id}+p^l \tilde{c}\in \op{GL}_n(\Z/p^{l+m+1})$ and $B=\op{Id}+p^{m}\tilde{d}\in \op{GL}_n(\Z/p^{l+m+1})$ lift $\op{Id}+p^lc$ and $\op{Id}+p^{m}d$ respectively. Assume without loss of generality that $l\leq m$. Since we are working mod-$p^{l+m+1}$, it follows that $(p^l\tilde{c})^{m+2}=0$ and $(p^m \tilde{d})^3=0$.
     We have that
\[\begin{split}
    ABA^{-1}B^{-1}=&(\operatorname{Id}+p^l\tilde{c})(\operatorname{Id}+p^{m}\tilde{d})(\operatorname{Id}+p^l\tilde{c})^{-1}(\operatorname{Id}+p^{m}\tilde{d})^{-1}\\
    =&(\operatorname{Id}+p^m \tilde{d})^{-1}+(\operatorname{Id}+p^l\tilde{c})p^m \tilde{d}(\operatorname{Id}+p^l\tilde{c})^{-1}(1+p^m \tilde{d})^{-1}\\
      =&(\operatorname{Id}-p^m \tilde{d}+(p^m\tilde{d})^2)\\&+(\operatorname{Id}+p^l\tilde{c})p^m \tilde{d}(\operatorname{Id}-p^l\tilde{c}+\cdots+(-1)^{m+1}(p^l\tilde{c})^{m+1})(\operatorname{Id}-p^m\tilde{d}+(p^m\tilde{d})^2)\\
      =&(\operatorname{Id}-p^m \tilde{d}+(p^m\tilde{d})^2)+(\operatorname{Id}+p^l\tilde{c})p^m \tilde{d}(\operatorname{Id}-p^l\tilde{c})(\operatorname{Id}-p^m\tilde{d})\\
     =&(\operatorname{Id}-p^m \tilde{d}+(p^m\tilde{d})^2)+(\operatorname{Id}+p^l\tilde{c})p^m \tilde{d}(\operatorname{Id}-p^l\tilde{c})-(p^{m} \tilde{d})^2\\
      =&(\operatorname{Id}-p^m \tilde{d}+(p^m\tilde{d})^2)+p^m \tilde{d}+p^{m+l}[c,d] -(p^{m} \tilde{d})^2\\
      =&\operatorname{Id}+p^{m+l}[c,d].\\
\end{split}\]
     \end{proof}
     The following Lemma will be applied to show that the representations we construct contain a finite index subgroup of $\op{SL}_n(\Z_p)$. 
     \begin{Lemma}\label{lemma27}
     Let $\rho:\op{G}_{\Q,\{p\}}\rightarrow \op{GL}_n(\Z_p)$ be a continuous Galois representation lifting $\bar{\rho}$. Let $m\geq 1$ be such that $\Phi_m(\rho)$ contains $\g$. Then we have the following:
     \begin{enumerate}
         \item\label{lemma27p1} $\Phi_k(\rho)$ contains $\g$ for $k\geq m$,
         \item\label{lemma27p2} the image of $\rho$ contains $\mathcal{U}_m$.
     \end{enumerate}
     \end{Lemma}
     \begin{proof} It is easy to check that $[\g,\g]=\g$. Part $\eqref{lemma27p1}$ follows from Lemma $\ref{lemma26}$. Let $H$ be the image of $\rho$. Since $\rho$ is continuous and $\op{G}_{\Q,\{p\}}$ is compact, it follows that $H$ is closed. For $l\geq 1$, let $H_l$ be the projection of $H$ to $\op{GL}_n(\Z/p^l)$. Since $H$ is closed, we may identify it with the inverse limit $\varprojlim_l H_l$. Thus for part $\eqref{lemma27p2}$, we only need to check that $H_k$ contains $\g$ for $k\geq m$. This follows from part $\eqref{lemma27p1}$. 
     \end{proof}
     \begin{Lemma} \label{lemma29}
     Let $\rho:\op{G}_{\Q,\{p\}}\rightarrow \op{GL}_n(\Z_p)$ be a continuous Galois representation lifting $\bar{\rho}$. Assume that $\Phi_1(\rho)$ contains an element $w:=\sum_i a_i e_{i,i}$ such that $a_1,\dots, a_n$ are all distinct. Furthermore, assume that it contains $e_{i,j}$ for all tuples $(i,j)$ such that $(i+j)$ is odd. Then we have the following:
     \begin{enumerate}
         \item $\Phi_4(\rho)$ contains $\g$,
         \item the image of $\rho$ contains $\mathcal{U}_4$.
     \end{enumerate}
     \end{Lemma}
     \begin{proof}
     First consider the case $n=2$. Lemma $\ref{lemma26}$ asserts that $[\Phi_1(\rho), \Phi_1(\rho)]$ is contained in $\Phi_2(\rho)$. The relations $[w, e_{1,2}]=(a_1-a_2) e_{1,2}$ and $[w, e_{2,1}]=(a_2-a_1) e_{2,1}$ imply that $e_{1,2}$ and $e_{2,1}$ are contained in $\Phi_2(\rho)$. The relation $[e_{1,2},e_{2,1}]=2(e_{1,1}-e_{2,2})$ implies that $e_{1,1}-e_{2,2}$ is also contained in $\Phi_2(\rho)$. Thus $\Phi_2(\rho)$ contains $\g$ and the conclusion follows from Lemma $\ref{lemma27}$.
     \par Consider the case $n>2$. Let $(i,j)$ be a tuple with $i\neq j$ and $i+j$ even. Since $n\geq 3$, we can pick $l$ such that $l+i$ and $l+j$ are both odd. The relation $e_{i,j}=[e_{i,l}, e_{l, j}]$ implies that $\Phi_2(\rho)$ contains $e_{i,j}$. Let $(i,j)$ be a pair with $i\neq j$ and $i+j$ is odd. The relation $[e_{i,j},w]=(a_j-a_i) e_{i,j}$ implies that $e_{i,j}$ is contained in $\Phi_2(\rho)$.
     \par 
     Since $[\Phi_1(\rho), \Phi_2(\rho)]$ is contained in $\Phi_3(\rho)$, the relation $[w,e_{i,j}]=(a_i-a_j)e_{i,j}$ implies that $\Phi_3(\rho)$ contains all $e_{i,j}$, where $(i,j)$ runs through pairs such that $i\neq j$. One more iteration of the same tells us that $\Phi_4(\rho)$ contains all $e_{i,j}$ where $(i,j)$ runs through pairs such that $i\neq j$. Next, we note that $[\Phi_2(\rho), \Phi_2(\rho)]$ is contained in $\Phi_4(\rho)$. The relation $[e_{i,j}, e_{j,i}]=e_{i,i}-e_{j,j}$ implies that all elements $e_{i,j}-e_{j,j}\in \mathfrak{t}$ are contained in $\Phi_4(\rho)$. We have thus shown that $\Phi_4(\rho)$ contains $\g$. The conclusion follows from Lemma $\ref{lemma27}$.
     \end{proof}
     \section{Proof of Main Results}
     \par Recall that $\bar{\rho}$ is the representation
\[\bar{\rho}=\left( {\begin{array}{cccc}
   \bar{\chi}^{k_1} & & & \\
    & \ddots& & \\
    & & \bar{\chi}^{k_{n-1}}& \\
    & & & \bar{\chi}^{k_{n}}
  \end{array} } \right):\op{G}_{\Q,\{p\}}\rightarrow \op{GL}_n(\F_p).\]
  \par Let $A$ be the Class group of $\Q(\mu_p)$ and let $\mathcal{C}$ denote the mod-$p$ class group $\mathcal{C}:=A\otimes \F_p$. The Galois group $\op{Gal}(\Q(\mu_p)/\Q)$ acts on $\mathcal{C}$ via the natural action. Since the order of $\op{Gal}(\Q(\mu_p)/\Q)$ is prime to $p$, it follows that $\mathcal{C}$ decomposes into eigenspaces 
  \[\mathcal{C}=\bigoplus_{i=0}^{p-2} \mathcal{C}(\bar{\chi}^i),\]
  where $\mathcal{C}(\bar{\chi}^i)=\{x\in \mathcal{C} \mid g\cdot x=\bar{\chi}^{i}(g) x \}$.
  \begin{Def}\label{index}
  The \textit{index of regularity} $e_p$ is the number of eigenspaces $\mathcal{C}(\bar{\chi}^i)$ which are non-zero.
  \end{Def}Note that Vandiver's conjecture predicts that $\mathcal{C}(\bar{\chi}^i)=0$ for $i$ even (cf. \cite[Chapter 8]{washington}). For a $\op{G}_{\Q,\{S\}}$-module $M$, which is a finite dimensional $\F_p$-vector space, we denote by $\Sh^i_{\{p\}}(M)$, the kernel of the restriction map 
  \[\Sh^i_{\{p\}}(M):=\op{ker} \left( H^i(\op{G}_{\Q,\{p\}}, M)\rightarrow H^i(\op{G}_{\Q_p}, M)\right).\] Let $M^*:=\op{Hom}(M, \mu_p)$, where $\mu_p$ is the Galois module of $p$-th roots of unity. Note that $\mu_p\simeq \F_p(\bar{\chi})$. Global duality for $\Sh$-groups states that there is a natural isomorphism $\Sh^2_{\{p\}}(M)\simeq \Sh^1_{\{p\}}(M^*)^{\vee}$.
  \begin{Lemma}\label{lemma31}
  For $0\leq i\leq p-2$,
  \begin{enumerate}
      \item the group $\Sh^1_{\{p\}}(\F_p(\bar{\chi}^i))$ is zero if $\mathcal{C}(\bar{\chi}^i)$ is zero,
      \item the group $\Sh^2_{\{p\}}(\F_p(\bar{\chi}^i))$ is zero if $\mathcal{C}(\bar{\chi}^{p-i})$ is zero.
  \end{enumerate}
  \end{Lemma}
  \begin{proof}
  Let $L$ be the subfield of the Hilbert Class field of $\Q(\mu_p)$ such that $\op{Gal}(L/\Q(\mu_p))$ is isomorphic to $\mathcal{C}$. Since the order of $\op{Gal}(\Q(\mu_p)/\Q)$ is prime to $p$, it follows that $H^j(\op{Gal}(\Q(\mu_p)/\Q), \F_p(\bar{\chi}^i))=0$ for $j=1,2$. It follows that the restriction map $H^1(\op{G}_{\Q},\F_p(\bar{\chi}^i))\rightarrow  \op{Hom}(\op{G}_{\Q(\mu_p)}, \F_p(\bar{\chi}^i))^{\op{Gal}(\Q(\mu_p)/\Q)}$ is an isomorphism. Via this isomorphism $\Sh^1_{\{p\}}(\F_p(\bar{\chi}^i))$ consists homomorphisms $\op{Hom}(\op{Gal}(L/\Q(\mu_p)), \F_p(\bar{\chi}^i))$ that are unramified outside $\{p\}$ and trivial when restricted to the prime of $\Q(\mu_p)$ above $p$. The conclusion of the first part follows. The second part follows from the first part and global duality.
  \end{proof}
  
  \begin{Th}\label{main2}
Let $k_1,\dots, k_n$ and $\bar{\rho}$ be as above. Assume that the following are satisfied:
\begin{enumerate}
\item $0<k_i <\frac{p-1}{2}$,
\item $k_i$ is odd for $i$ even and even for $i$ odd,
\item $\bar{\chi}^{k_i-k_j}$ is not equal to $\bar{\chi}$.
    \item The characters $\bar{\chi}^{k_i-k_j}$ for $i\neq j$ are all distinct.
    \item For $(i,j)$ such that $i\neq j$, we have that $\mathcal{C}(\bar{\chi}^{p-(k_i-k_j)})=0$.
\end{enumerate}
Then there exists a continuous lift $\rho:\op{G}_{\Q,\{p\}}\rightarrow \op{GL}_n(\Z_p)$ of $\bar{\rho}$ such that the image of $\rho$ contains $\mathcal{U}_4$.
  \end{Th}
  \begin{proof}
  First, we exhibit a characteristic zero lift of $\bar{\rho}$ which is unramified outside $\{p\}$. We show that the unramified outside $\{p\}$ deformation problem is unobstructed, i.e., $H^2(\op{G}_{\Q,\{p\}}, \g)=0$. Note that $H^2(\op{G}_{\Q_p}, \g)$ decomposes into
  \[H^2(\op{G}_{\Q_p}, \g)=H^2(\op{G}_{\Q_p}, \mathfrak{t})\oplus \left(\bigoplus_{(i,j)} H^2(\op{G}_{\Q_p}, \F_q(\bar{\chi}^{k_i-k_j}))\right),\]
  where $(i,j)$ runs through pairs for which $i\neq j$.
  By local duality, we have that
  \[H^2(\op{G}_{\Q_p}, \mathfrak{t})\simeq H^0(\op{G}_{\Q_p}, \mathfrak{t}^*)^{\vee}, \text{ and }H^2(\op{G}_{\Q_p}, \F_q(\bar{\chi}^{k_i-k_j}))\simeq H^0(\op{G}_{\Q_p}, \F_q(\bar{\chi}^{p-(k_i-k_j)}))^{\vee}.\]
  By assumption, $\bar{\chi}^{k_i-k_j}\neq \bar{\chi}$. As a result, we have that $H^0(\op{G}_{\Q_p}, \F_q(\bar{\chi}^{p-(k_i-k_j)}))=0$. On the other hand, the Galois action on $\mathfrak{t}$ is trivial and the dual acquires a twist by $\bar{\chi}$, hence, $H^0(\op{G}_{\Q_p}, \mathfrak{t}^*)=0$. Thus, the local cohomology group $H^2(\op{G}_{\Q_p}, \g)$ is zero and hence,
  \[H^2(\op{G}_{\Q,\{p\}}, \g)=\Sh^2_{\{p\}}(\g).\]
  By global duality, we have that
  \[\Sh^2_{\{p\}}( \mathfrak{t})\simeq \Sh^1_{\{p\}}( \mathfrak{t}^*)^{\vee}.\]
  It is a well known fact that $\mathcal{C}(\bar{\chi})$ is zero (cf. Proposition 6.16 of \cite{washington}). It follows (from Lemma $\ref{lemma31}$) that $\Sh^1_{\{p\}}(\F_p(\bar{\chi}))$ is zero, and thus, $\Sh^1_{\{p\}}(\mathfrak{t}^*)$ is zero. By assumption, $\mathcal{C}(\bar{\chi}^{p-(k_i-k_j)})$ is zero, and hence, by Lemma $\ref{lemma31}$, \[\Sh^2_{\{p\}} (\F_q(\bar{\chi}^{k_i-k_j}))=0.\] It has thus been shown that $H^2(\op{G}_{\Q,\{p\}}, \g)=0$.
  \par Recall that $\chi_2$ is $\chi\mod{p^2}$, let $\rho_2'$ be the lift\[\rho_2'=\left( {\begin{array}{cccc}
   \chi_2^{k_1} & & & \\
    & \ddots& & \\
    & & \chi_2^{k_{n-1}}& \\
    & & & \chi_2^{k_{n}}
  \end{array} } \right):\op{G}_{\Q,\{p\}}\rightarrow \op{GL}_n(\Z/p^2).\] Let $(i,j)$ be a pair such that $i+j$ is odd. Since $H^2(\op{G}_{\Q,\{p\}}, \F_p(\bar{\chi}^{k_i-k_j}))=0$ and $H^0(\op{G}_{\Q_{\infty}}, \F_p(\bar{\chi}^{k_i-k_j}))=0$, it follows from the Global Euler characteristic formula (see \cite[Theorem 8.7.4]{NSW}) that $H^1(\op{G}_{\Q,\{p\}}, \F_p(\bar{\chi}^{k_i-k_j}))$ is one-dimensional. Let $f_{i,j}$ be a generator of $H^1(\op{G}_{\Q,\{p\}}, \F_p(\bar{\chi}^{k_i-k_j}))$. Let $F\in H^1((\op{G}_{\Q,\{p\}}, \g)$ be the sum of all $f_{i,j}$ where $(i,j)$ ranges over all $i\neq j$ such that $i+j$ is odd. Let $\rho_2$ be the twist $(\op{Id}+p F)\rho_2':\op{G}_{\Q,\{p\}}\rightarrow \op{GL}_n(\Z/p^2)$. Since $H^2(\op{G}_{\Q,\{p\}}, \g)=0$, it follows from Lemma $\ref{lemma25}$ that $\rho_2$ lifts to a characteristic zero Galois representation $\rho:\op{G}_{\Q,\{p\}}\rightarrow \op{GL}_n(\Z_p)$.
  \par In order to show that the image of $\rho$ contains $\mathcal{U}_4$, it suffices (by Lemma $\ref{lemma29}$) to show that $\Phi_1(\rho)$ contains:
  \begin{itemize}\item $e_{i,j}$ for all tuples $(i,j)$ such that $i+j$ is odd,
  \item an element $w=\sum a_i e_{i,i}$ in $\mathfrak{t}$ such that the $a_i$ are distinct.
  \end{itemize}The image of $\bar{\rho}$ is prime to $p$ and as a result, any Galois-submodule $M$ of $\g$ decomposes into
  \[M=M_1\oplus \left(\bigoplus_{(i,j)} M_{\bar{\chi}^{k_i-k_j}}\right)\]where $M_{\bar{\chi}^{k_i-k_j}}$ is the submodule 
  \[M_{\bar{\chi}^{k_i-k_j}}:=\{x\in M\mid g\cdot x=\bar{\chi}^{k_i-k_j}(g)x\}\]
  and $M_1$ the $\op{G}_{\Q}$-invariant submodule. Note that it is assumed that all characters $\bar{\chi}^{k_i-k_j}$ are distinct for $i<j$. It follows from the bounds on $k_i$ that all characters $\bar{\chi}^{k_i-k_j}$ are distinct for all tuples $(i,j)$ with $i\neq j$. It is also clear that none of these is the trivial character. As a result, the above decomposition makes sense and $M_{\bar{\chi}^{k_i-k_j}}$, if non-zero, is the one-dimensional space generated by $e_{i,j}$. Since the order of $\Q(\mu_p)$ over $\Q$ is prime to $p$, it follows that
  \[H^1(\op{Gal}(\Q(\mu_p)/\Q), \F_p(\bar{\chi}^{k_i-k_j}))=0.\]It follows from the inflation restriction sequence that the restriction of $f_{i,j}$ to $\op{G}_{\Q(\mu_p)}$ is non-zero. Hence, there exists $g\in \op{ker} \bar{\rho}$ such that $f_{i,j}(g)\neq 0$. The element $\rho_2(g)\in \Phi_1(\rho)$ has non-zero $e_{i,j}$-component. It follows from the decomposition 
  \[\Phi_1(\rho)=\Phi_1(\rho)^{\op{G}_{\Q}}\oplus \left(\bigoplus_{(i,j)} \Phi_1(\rho)_{\bar{\chi}^{k_i-k_j}}\right)\]
  that $e_{i,j}\in \Phi_1(\rho)$ for all tuples $(i,j)$ for which $i+j$ is odd. Note that the cyclotomic character $\chi$ induces an isomorphism 
  \[\chi:\op{Gal}(\Q(\mu_{p^{\infty}})/\Q(\mu_p))\xrightarrow{\sim} 1+p\Z_p.\]Let $\gamma\in \op{G}_{\Q(\mu_p)}$ be chosen such that $\chi(\gamma)=1+p$. With respect to the identification of $1+pX\in \Phi_1(\rho)$ with $X\in \g$, the element $\rho_2(\gamma)$ coincides with $w:=\sum k_i e_{i,i}$. We have thus shown that $\Phi_1(\rho)$ satisfies the required conditions, and this completes the proof.
  
  \end{proof}
  
  \begin{proof}(proof of Theorem $\ref{main}$)
  Consider the $t:=n+2e$ numbers $m_1,\dots, m_t$, where $m_j=2^{j+1}+\epsilon_j$ and \[\epsilon_j:=\begin{cases}
  0\text{ if }j \text{ is odd.}\\
  1\text{ if }j \text{ is even.}
  \end{cases}\]
  Note that $4= m_1<m_2<\dots< m_t< \frac{p-1}{2}$. Suppose that $(i,j)$ and $(k,l)$ are such that $i\neq j$, $k\neq l$ and 
  \[m_i-m_j\equiv m_k-m_l\mod{p-1}.\]Then we show that $i=k$ and $j=l$. Since $|m_i-m_j| $ and $|m_k-m_l|$ are less than $\frac{p-1}{2}$, we have that $m_i-m_j= m_k-m_l$. Assume without loss of generality that $i>j$, and thus $m_i-m_j>0$. This implies that $m_k>m_l$ and thus $k>l$. We have that 
  \[2^{i+1}-2^{j+1}=2^{k+1}-2^{l+1}+\alpha,\]where $-2\leq \alpha\leq 2$. Since $i,j,k,l\geq 1$, we deduce that $4$ divides $\alpha$, and thus $\alpha=0$. It thus suffices to show that $i=k$. Suppose not, assume without loss of generality that $i>k$. Then we have 
  \[2^{i+1}=2^{j+1}+2^{k+1}-2^{l+1}\leq 2^{i}+2^{i}-2^{l+1}<2^{i+1}.\]Thus, it follows that $(i,j)=(k,l)$. It follows that the characters $\bar{\chi}^{p-(m_i-m_j)}$ are all distinct as $(i,j)$ ranges over all tuples such that $i\neq j$. Let $\mathcal{S}_i$ be the set of characters $\bar{\chi}^{p-(m_i-m_j)}$ as $j$ ranges from $1$ to $t$ such that $j\neq i$. Since the index of regularity $e_p$ is less than or equal to $e$, it follows that there is a subset $\{i_1,\dots, i_{n+e}\}$ of $\{1,\dots, t\}$ such that for each character $\beta \in \bigcup \mathcal{S}_{i_j}$, the eigenspace $\mathcal{C}(\beta)=0$. There is a subset of $n$ numbers $\{k_1,\dots, k_n\}$ of $\{i_1,\dots, i_{n+e}\}$ such that $k_i$ is odd if $i$ is odd and even if $i$ is even. Set $a_i$ to be equal to $m_{k_i}$. Note that $a_i$ is even for $i$ odd and odd for $i$ even. Moreover, the characters $\bar{\chi}^{p-(a_i-a_j)}$ are all distinct and $\mathcal{C}(\bar{\chi}^{p-(a_i-a_j)})=0$. It is clear from the definition of the original sequence $\{m_i\}$ that $a_i-a_j$ is not equal to $1$. The result follows from Theorem \ref{main2}. In fact, there are infinitely many Galois representations since there are infinitely many choices of \[\psi:\op{G}_{\Q,\{p\}}\rightarrow \op{GL}_1(\Z_p)\] such that $\psi_2=(\tilde{\tau})_2$.
  \end{proof}


\begin{thebibliography}{1}
\bibitem{buhler}Buhler, J. P., R. E. Crandall, and R. W. Sompolski. "Irregular primes to one million." mathematics of computation 59.200 (1992): 717-722.

\bibitem{raycornut}Cornut, Christophe, and Jishnu Ray. "Generators of the pro-p Iwahori and Galois representations." International Journal of Number Theory 14.01 (2018): 37-53.

\bibitem{FonaineMazur}Fontaine, Jean-Marc, and Barry Mazur. "Geometric Galois representations." Elliptic curves, modular forms, and Fermat's last theorem (Hong Kong, 1993), Ser. Number Theory, I (1995): 41-78.
\bibitem{greenberg}Greenberg, Ralph. "Galois representations with open image." Annales math\'ematiques du Qu\'ebec 40.1 (2016): 83-119.
\bibitem{hamblenramakrishna} Hamblen, Spencer and Ramakrishna, Ravi. "Deformations of certain reducible {G}alois representations {II}". In:Amer. J. Math.130.4(2008),pp.913-944.

\bibitem{hart}Hart, William, David Harvey, and Wilson Ong. "Irregular primes to two billion." Mathematics of Computation 86.308 (2017): 3031-3049.
\bibitem{Mazur}Mazur, Barry. "An introduction to the deformation theory of Galois representations." Modular forms and Fermat's last theorem. Springer, New York, NY, 1997. 243-311.
\bibitem{NSW}Neukirch, J\"urgen, Alexander Schmidt, and Kay Wingberg. Cohomology of number fields. Vol. 323. Springer Science and Business Media, 2013.


\bibitem{Ram2} Ramakrishna, Ravi. "Lifting Galois representations." Inventiones mathematicae 138.3 (1999): 537-562.

\bibitem{RaviFM}Ramakrishna, Ravi. "Deforming Galois representations and the conjectures of Serre and Fontaine-Mazur." Annals of mathematics 156.1 (2002): 115-154.
\bibitem{tang}Tang, Shiang. "Algebraic monodromy groups of G-valued l-adic Galois representations." Algebra and Number Theory.
\bibitem{washington} L. Washington,
Introduction to Cyclotomic Fields
, 2nd edition,
Grad. Texts in Math.
, vol. 83, Springer-
Verlag, Berlin, 1997. \href{https://doi.org/10.1007/978-1-4612-1934-7}{https://doi.org/10.1007/978-1-4612-1934-7}.


\end{thebibliography}
\end{document}